\documentclass[11pt]{amsart}

\usepackage{amssymb,amsmath,color,hyperref}
\usepackage[mathscr]{eucal}

\theoremstyle{plain}
\newtheorem{thm}{Theorem}[section]

\newtheorem{lem}[thm]{Lemma}

\newtheorem{prop}[thm]{Proposition}

\theoremstyle{definition}
\newtheorem{defn}[thm]{Definition}

\theoremstyle{remark}

\newtheorem*{remark*}{Remark}

\numberwithin{equation}{section}

\newcommand{\cK}{{\mathcal{K}}}

\newcommand{\cP}{{\mathcal{P}}}

        \newcommand{\field}[1]{{\mathbb{#1}}}
        \newcommand{\NN}{\field{N}}
        \newcommand{\ZZ}{\field{Z}}
        
        \newcommand{\RR}{\field{R}}
        \newcommand{\CC}{\field{C}}

\allowdisplaybreaks

\begin{document}

\title[Berezin-Toeplitz quantization]{Berezin-Toeplitz quantization asssociated with higher Landau levels of the Bochner Laplacian}

\author[Y. A. Kordyukov]{Yuri A. Kordyukov}
\address{Institute of Mathematics, Ufa Federal Research Centre, Russian Academy of Sciences, 112~Chernyshevsky str., 450008 Ufa, Russia and Kazan Federal University, 18 Kremlyovskaya str., 420008 Kazan,  Russia} \email{yurikor@matem.anrb.ru}

\thanks{Partially supported by the development program of the Regional Scientific and Educational Mathematical Center of the Volga Federal District, agreement N 075-02-2020-1478.}



\begin{abstract}
In  this paper, we construct a family of Berezin-Toeplitz type quantizations of a compact symplectic manifold. For this, we choose a Riemannian metric on the manifold such that the associated Bochner Laplacian has the same local model at each point (this is slightly more general than in almost-K\"ahler quantization). Then the spectrum of the Bochner Laplacian on high tensor powers $L^p$ of the prequantum line bundle $L$ asymptotically splits into clusters of size ${\mathcal O}(p^{3/4})$ around the points $p\Lambda$, where $\Lambda$ is an eigenvalue of the model operator (which can be naturally called a Landau level). We develop the Toeplitz operator calculus with the quantum space, which is the eigenspace of the Bochner Laplacian corresponding to the eigebvalues frrom the cluster. We show that it provides a Berezin-Toeplitz quantization. If the cluster corresponds to a Landau level of multiplicity one, we obtain an algebra of Toeplitz operators and a formal star-product. For the lowest Landau level, it recovers the almost K\"ahler quantization.
\end{abstract}

\date{December 28, 2020}

 \maketitle

\section{Introduction}
The main goal of our paper is to construct a family of Berezin-Toeplitz quantizations based on appropriate eigenspaces of  the Bochner Laplacian under certain condition on the Riemannian metric on the symplectic manifold (which is slightly more general than the almost-K\"ahler one).  More precisely, let $(X,\mathbf B)$ be a closed symplectic manifold of dimension $2n$. Assume that there exists a Hermitian line bundle $(L,h^L)$ on $X$ with a Hermitian connection $\nabla^L$ such that 
\begin{equation}\label{e:def-omega}
\mathbf B=iR^L, 
\end{equation} 
where $R^L$ is the curvature of the connection $\nabla^L $ defined as $R^L=(\nabla^L)^2$. 

Let $g$ be a Riemannian metric on $X$ and $(E,h^E)$ be a Hermitian vector bundle of rank $r$ on $X$ with a Hermitian connection $\nabla^E$. For any $p\in \NN$, let $L^p:=L^{\otimes p}$ be the $p$th tensor power of $L$ and let
\[
\nabla^{L^p\otimes E}: {C}^\infty(X,L^p\otimes E)\to
{C}^\infty(X, T^*X \otimes L^p\otimes E)
\] 
be the Hermitian connection on $L^p\otimes E$ induced by $\nabla^{L}$ and $\nabla^E$. Consider the induced Bochner Laplacian $\Delta^{L^p\otimes E}$ acting on $C^\infty(X,L^p\otimes E)$ by
\begin{equation}\label{e:def-Bochner}
\Delta^{L^p\otimes E}=\big(\nabla^{L^p\otimes E}\big)^{\!*}\,
\nabla^{L^p\otimes E},
\end{equation} 
where $\big(\nabla^{L^p\otimes E}\big)^{\!*}: {C}^\infty(X,T^*X\otimes L^p\otimes E)\to
{C}^\infty(X,L^p\otimes E)$ is the formal adjoint of  $\nabla^{L^p\otimes E}$. 

For an arbitrary $x\in X$, one can introduce a second order differential operator acting on $C^\infty(T_{x}X, E_{x})$ (the model operator), which is obtained from the Bochner Laplacian $\Delta^{L^p\otimes E}$ by freezing coefficients at $x$ (see \eqref{e:DeltaL0p} below and \cite{Kor20} for more details). It is the Bochner Laplacian on a constant curvature Hermitian line bundle over the Euclidean space $T_{x}X$. It can be also considered as the magnetic Laplacian with constant magnetic field. We consider the skew-adjoint operator $B_x : T_xX\to T_xX$ such that 
\[
\mathbf B_x(u,v)=g(B_xu,v), \quad u,v\in T_xX. 
\]
Its eigenvalues have the form $\pm i a_j(x), j=1,\ldots,n,$ with $a_j(x)>0$. 
The spectrum of the model operator consists of eigenvalues of the form
$\sum_{j=1}^n(2k_j+1)a_j(x)$ with $(k_1,\cdots,k_n)\in\ZZ_+^n$. Each eigenvalue has
 infinite multiplicity and can be called a Landau level.

We assume that the functions $a_j$ can be chosen to be constants:
\begin{equation}\label{e:aj-constant}
a_j(x)\equiv a_j, \quad x\in X, \quad  j=1,\ldots,n.
\end{equation}
This is a condition on the Riemannian metric $g$, which can be satisfied for any symplectic manifold $X$. In this case, that the spectrum of the model operator is independent of $x$ and coincides with  
the countable discrete set 
\begin{equation}\label{e:def-Sigmax}
\Sigma:=\left\{\Lambda_{\mathbf k}:=\sum_{j=1}^n(2k_j+1) a_j\,:\, \mathbf k=(k_1,\cdots,k_n)\in\ZZ_+^n\right\}.
\end{equation}
If $J=\frac{1}{2\pi}B$ is an almost-complex structure (the almost K\"ahler case), then $a_j=2\pi,  j=1,\ldots,n$ and
\begin{equation}\label{e:Kaehler}
\Sigma=\left\{2\pi (2k+n)\,:\, k\in\ZZ_+\right\}.
\end{equation}
As shown in \cite{Kor20} (see also \cite{FT}), for any $K>0$, there exists $c>0$ such that for any $p\in \NN$ the spectrum of $\Delta^{L^p\otimes E}$ in the interval  $[0,K]$  is  contained in the $cp^{3/4}$-neighborhood of $p\Sigma$. In other words, the spectrum of $\Delta^{L^p\otimes E}$ asymptotically splits into clusters  around $p\Sigma$ of size ${\mathcal O}(p^{3/4})$.

Next, we fix one of these clusters associated with $\Lambda\in \Sigma$ and develop the Toeplitz operator calculus associated with the eigenspace of the Bochner Laplacian corresponding to eigenvalues from this cluster. Consider an interval $I=(\alpha,\beta)$ such that $(\alpha,\beta)\cap \Sigma=\{\Lambda\}$. By the above mentioned fact, there exist $\mu_0>0$ and $p_0\in \NN$ such that for any $p>p_0$ 
\[
\sigma(\Delta^{L^p\otimes E})\subset (-\infty, p(\Lambda -\mu_0)) \cup (p\alpha,p\beta) \cup (p(\Lambda+\mu_0), \infty).
\] 
The spectral projection of the operator $\Delta^{L^p\otimes E}$ associated with $(p\alpha,p\beta)$ is independent of the choice of $I$ and will be denoted by $P_{p,\Lambda}$. 

For $f\in C^\infty(X,\operatorname{End}(E))$, we define the  associated Toeplitz operator to be the sequence of bounded linear operators
\[
T_{f,p}=P_{p,\Lambda}fP_{p,\Lambda}: L^2(X,L^p\otimes E)\to L^2(X,L^p\otimes E), \quad p\in \NN.
\] 

\begin{thm}\label{t:comm}
Let $f,g\in C^\infty(X,\operatorname{End}(E))$. Then, for the product of the Toeplitz operators $\{T_{f,p}\}$ and $\{T_{g,p}\}$, we have
\begin{equation}\label{e:TfpTgp-prod}
T_{f,p}T_{g,p}=T_{fg,p}+\mathcal O(p^{-1}). 
\end{equation}
Moreover, if $f,g\in C^\infty(X)$, then, for the commutator of the operators $\{T_{f,p}\}$ and $\{T_{g,p}\}$, we have
\begin{equation}\label{e:TfpTgp-comm}
[T_{f,p}, T_{g,p}]=i p^{-1}T_{\{f,g\},p}+ \mathcal O(p^{-1/2}), 
\end{equation}
where $\{f,g\}$ is the Poisson bracket on the symplectic manifold $(X,\mathbf B)$.
\end{thm}

Thus, the Toeplitz operators provide a Berezin-Toeplitz quantization for the compact symplectic manifold $(X,\mathbf B)$. The limit $p\to +\infty$ for Toeplitz operators can be thought of as a semiclassical limit, with semiclassical parameter $\hbar=\frac{1}{p}\to 0$. Theorem~\ref{t:algebra} shows that this quantization has a correct semiclassical limit. 

In the case when the set $\mathcal K_\Lambda:=\{\mathbf k\in \ZZ^n_+ : \Lambda_{\mathbf k}=\Lambda\}$ consists of a single element, we construct the algebra of Toeplitz operators associated with $\Lambda$. 

\begin{defn}\label{d:Toeplitz0}
A Toeplitz operator is a sequence $\{T_p\}=\{T_p\}_{p\in \mathbb N}$ of bounded linear operators $T_p : L^2(X,L^p\otimes E)\to L^2(X,L^p\otimes E)$, satisfying the following conditions.
\begin{description}
\item[(i)] For any $p\in \mathbb N$, we have 
\[
T_p=P_{p,\Lambda}T_pP_{p,\Lambda}. 
\]
\item[(ii)] There exists a sequence $g_l\in C^\infty(X,\operatorname{End}(E))$ such that 
\[
T_p=P_{p,\Lambda}\left(\sum_{l=0}^\infty p^{-l}g_l\right)P_{p,\Lambda}+\mathcal O(p^{-\infty}),
\]
i.e. for any natural $k$ there exists $C_k>0$ such that 
\[
\left\|T_p-P_{p,\Lambda}\left(\sum_{l=0}^k p^{-l}g_l\right)P_{p,\Lambda}\right\|\leq C_kp^{-k-1}.
\]
\end{description}
  \end{defn}
  
\begin{thm}\label{t:algebra}
Assume that $\mathcal K_\Lambda$ consists of a single element. Then, for any $f,g\in C^\infty(X,\operatorname{End}(E))$,  the product of the Toeplitz operators $\{T_{f,p}\}$ and $\{T_{g,p}\}$ is a Toeplitz operator in the sense of Definition \ref{d:Toeplitz0}. More precisely, it admits the asymptotic expansion 
\begin{equation}\label{e:TfTg-exp}
T_{f,p}T_{g,p}=\sum_{r=0}^\infty p^{-r}T_{C_r(f,g),p}+\mathcal O(p^{-\infty}), 
\end{equation}
with some $C_r(f,g)\in C^\infty(X,\operatorname{End}(E))$, where the $C_r$ are bidifferential operators. In particular, $C_0(f,g)=fg$ and, for $f,g\in C^\infty(X)$, we have
\begin{equation}\label{e:C1fg-gf}
C_1(f,g)-C_1(f,g)=i\{f,g\}.
\end{equation}
\end{thm}

The idea to use Toeplitz operators for quantization of K\"ahler manifolds was suggested by Berezin in \cite{Berezin}. We refer the reader to \cite{Ali,Englis,ma:ICMtalk,Schlich10} for some recent surveys on Berezin-Toeplitz and geometric quantization. For a general compact K\"ahler manifold, the Berezin-Toeplitz quantization was constructed by Bordemann-Meinrenken-Schli\-chen\-maier \cite{BMS94}, using the theory of Toeplitz structures of Boutet de Monvel and Guillemin \cite{BG}. In this case, the quantum space is the space of holomorphic sections of tensor powers of the prequantum line bundle over the K\"ahler manifold. For an arbitrary symplectic manifold, Guillemin and Vergne suggested to use the kernel of the spin$^c$ Dirac operator as a quantum space. The corresponding Berezin-Toeplitz quantization was developed by Ma-Marinescu \cite{ma-ma:book,ma-ma08}. It is based on the asymptotic expansion of the Bergman kernel outside the diagonal obtained by Dai-Liu-Ma \cite{dai-liu-ma}. Another candidate for the quantum space was suggested by Guillemin-Uribe \cite{Gu-Uribe}. It is the space of eigensections of the renormalized Bochner Laplacian corresponding to eigenvalues localized near the origin. In this case, the Berezin-Toeplitz quantization was recently constructed in \cite{ioos-lu-ma-ma,Kor18}, based on Ma-Marinescu work: the Bergman kernel expansion from \cite{ma-ma08a} and Toeplitz calculus developed in \cite{ma-ma08} for spin$^c$ Dirac operator and K\"ahler case (also with an auxiliary bundle). We note also that Charles \cite{charles16} proposed recently another approach to quantization of symplectic manifolds and Hsiao-Marinescu \cite{HM} constructed a Berezin-Toeplitz quantization for eigensections of small eigenvalues in the case of complex manifolds. 

In our paper, we follow the approach to Toeplitz operator calculus developed in \cite{ma-ma:book,ma-ma08,ioos-lu-ma-ma,Kor18}. Asymptotic expansions of the kernels of spectral projections, which we need in our case, are proved in \cite{Kor20}. When $\Lambda=\Lambda_0$ is the lowest Landau level, our results are reduced to the results obtained in \cite{ioos-lu-ma-ma,Kor18}, which hold for any Riemannian metric $g$ (not necessarily satisfying the condition \eqref{e:aj-constant}). We mention that, in two simultaneous papers \cite{charles20a,charles20b}, Charles studies the same subject, using the methods of \cite{charles16}. 

There are several papers devoted to Toeplitz operators acting on spectral subspaces of the Landau Hamiltonian in $\RR^{2n}$ (see, for instance, \cite{BPR04,FP06,MR03,PRV13,RW02,RT08,RT09} and references therein). For constant magnetic fields, such operators are related with the Toeplitz operators acting on Bargmann-Fock type spaces of polyanalytic functions (see, for instance, \cite{AF14,AG10,EZ17,G10,keller-luef,roz-vas19,V00} and references therein). In particular, in \cite{keller-luef}, quantization schemes defined by polyanalytic Toeplitz operators are discussed. 

The paper is organized as follows. In Section~\ref{s:algebraA}, we introduce an algebra $\mathfrak A$ of integral operators on $X$ defined in terms of  conditions on their smooth Schwartz kernels. In Section~\ref{s:Toeplitz-in-A}, we show that Toeplitz operators in the sense of Definition \ref{d:Toeplitz0} belong to $\mathfrak A$. In Section \ref{s:Thm1}, using the results of the previous sections, we prove the first part of Theorem \ref{t:comm} and reduce the proof of its second part to a similar statement in the Euclidean case. The proof of this statement is given in Section~\ref{s:model}. In Section~\ref{s:charact}, we prove that the set of Toeplitz operators coincides with the algebra $\mathfrak A$ in the case when $\mathcal K_\Lambda$ consists of a single element, which gives a characterization of Toeplitz operators in terms of their Schwartz kernels in the form introduced in \cite[Theorem 4.9]{ma-ma08}. Using Theorems~\ref{t:characterization} and \ref{t:comm},  we easily complete the proof of Theorem \ref{t:algebra}.

This work was started as a joint project with L. Charles, but later we decided to work on our approaches separately. I would like to thank Laurent for his collaboration. 

\section{Algebra of integral operators}\label{s:algebraA}
In this section, we introduce an algebra $\mathfrak A$ of integral operators on $X$ defined in terms of  conditions on their smooth Schwartz kernels. Our motivation comes from the description of Toeplitz operators in terms of their Schwartz kernels introduced in \cite{ma-ma:book,ma-ma08}. Later, we will show that Toeplitz operators in the sense of Definition \ref{d:Toeplitz0} belong to this algebra, and, if $\mathcal K_\Lambda$ consists of a single element, the set of Toeplitz operators coincides with $\mathfrak A$. 

We introduce normal coordinates near an arbitrary point $x_0\in X$. 
We denote by $B^{X}(x_0,r)$ and $B^{T_{x_0}X}(0,r)$ the open balls in $X$ and $T_{x_0}X$ with center $x_0$ and radius $r$, respectively. Let $r_X>0$ be the injectivity radius of $X$. We identify $B^{T_{x_0}X}(0,r_X)$ with $B^{X}(x_0,r_X)$ via the exponential map $\exp^X_{x_0}: T_{x_0}X \to X$. Furthermore, we choose trivializations of the bundles $L$ and $E$ over $B^{X}(x_0,r_X)$,   identifying their fibers $L_Z$ and $E_Z$ at $Z\in B^{T_{x_0}X}(0,r_X)\cong B^{X}(x_0,r_X)$ with the spaces  $L_{x_0}$ and $E_{x_0}$ by parallel transport with respect to the connections $\nabla^L$ and $\nabla^E$ along the curve $\gamma_Z : [0,1]\ni u \to \exp^X_{x_0}(uZ)$. Denote by $\nabla^{L^p\otimes E}$ and $h^{L^p\otimes E}$ the connection and the Hermitian metric on the trivial bundle with fiber $(L^p\otimes E)_{x_0}$ induced by these trivializations.  

We choose an orthonormal base $\{e_j : j=1,\ldots,2n\}$ in $T_{x_0}X$ such that  
\begin{equation}\label{e:obase}
B_{x_0}e_{2k-1}=a_k e_{2k}, \quad B_{x_0}e_{2k}=-a_ke_{2k-1},\quad k=1,\ldots,n. 
\end{equation}
Thus, we have 
\begin{equation}\label{e:Bx0}
\mathbf B_{x_0}=\sum_{k=1}^{n}a_kdZ_{2k-1}\wedge dZ_{2k}. 
\end{equation}

We introduce a coordinate chart $\gamma_{x_0} : B(0,c)\subset  \RR^{2n}\to X$  defined on the ball $B(0,c):=\{Z\in \RR^{2n} : |Z|<c\}$ with some $c\in (0,r_X)$, which is given by the restriction of the exponential map $\exp_{x_0}^X$ composed with the linear isomorphism $\mathbb R^{2n}\to T_{x_0}X$ determined by the base $\{e_j \}$.

Let $dv_{TX}$ denote the Riemannian volume form of the Euclidean space $(T_{x_0}X, g_{x_0})$. We define a smooth function $\kappa$ on $B^{T_{x_0}X}(0,a^X)\cong B^{X}(x_0,a^X)$ by the equation
\[
dv_{X}(Z)=\kappa(Z)dv_{TX}(Z), \quad Z\in B^{T_{x_0}X}(0,a^X). 
\] 

Let $\{\Xi_p\}$ be a sequence  of linear operators $\Xi_p : L^2(X,L^p\otimes E)\to L^2(X,L^p\otimes E)$ with smooth kernel $\Xi_p(x,x^\prime)$ with respect to $dv_X$. Consider the fiberwise product $TX\times_X TX=\{(Z,Z^\prime)\in T_{x_0}X\times T_{x_0}X : x_0\in X\}$. Let $\pi : TX\times_X TX\to X$ be the natural projection given by  $\pi(Z,Z^\prime)=x_0$.  The kernel $\Xi_p(x,x^\prime)$ induces a smooth section $\Xi_{p,x_0}(Z,Z^\prime)$ of the vector bundle $\pi^*(\operatorname{End}(E))$ on $TX\times_X TX$ defined for all $x_0\in X$ and $Z,Z^\prime\in T_{x_0}X$ with $|Z|, |Z^\prime|<r_X$. 

Denote by $\mathcal P$ the Bergman kernel in $\mathbb R^{2n}$ given by 
\begin{equation}
\label{e:Bergman}
\mathcal P(Z,Z^\prime)=\frac{1}{(2\pi)^n}\prod_{j=1}^na_j \exp\left(-\frac 14\sum_{k=1}^na_k(|z_k|^2+|z_k^\prime|^2- 2z_k\bar z_k^\prime) \right) .
\end{equation}
We will use the same notation for the corresponding scalar function $\mathcal P(Z,Z^\prime)=\mathcal P(Z,Z^\prime)\operatorname{Id}_{E_{x_0}}$ on $T_{x_0}X\times T_{x_0}X$ with values in $\operatorname{End}(E_{x_0})$.

\begin{defn}[\cite{ma-ma:book,ma-ma08}]
We say that 
\begin{equation}
\label{e:full-off}
p^{-n}\Xi_{p,x_0}(Z,Z^\prime)\cong \sum_{r=0}^k(Q_{r,x_0}\mathcal P)(\sqrt{p}Z,\sqrt{p}Z^\prime)p^{-\frac{r}{2}}+\mathcal O(p^{-\frac{k+1}{2}})
\end{equation}
with some $Q_{r,x_0}\in \operatorname{End}(E_{x_0})[Z,Z^\prime]$, $0\leq r\leq k$,  depending smoothly on $x_0\in X$, if there exist $\varepsilon^\prime\in (0,r_X]$ and $C_0>0$ with the following property:
for any $l\in \mathbb N$, there exist $C>0$ and $M>0$ such that for any $x_0\in X$, $p\geq 1$ and $Z,Z^\prime\in T_{x_0}X$, $|Z|, |Z^\prime|<\varepsilon^\prime$, we have 
\begin{multline*}
\Bigg|p^{-n}\Xi_{p,x_0}(Z,Z^\prime)\kappa^{\frac 12}(Z)\kappa^{\frac 12}(Z^\prime) -\sum_{r=0}^k(Q_{r,x_0}\mathcal P)(\sqrt{p} Z, \sqrt{p}Z^\prime)p^{-\frac{r}{2}}\Bigg|_{\mathcal C^{l}(X)}\\ 
\leq Cp^{-\frac{k+1}{2}}(1+\sqrt{p}|Z|+\sqrt{p}|Z^\prime|)^M\exp(-\sqrt{C_0p}|Z-Z^\prime|)+\mathcal O(p^{-\infty}).
\end{multline*}
\end{defn}

Here $\mathcal C^{m^\prime}(X)$ is the $\mathcal C^{m^\prime}$-norm for the parameter $x_0\in X$. We say that $G_p=\mathcal O(p^{-\infty})$ if for any $l, l_1\in \mathbb N$, there exists $C_{l,l_1}>0$ such that $\mathcal C^{l_1}$-norm of $G_p$ is estimated from above by $C_{l,l_1}p^{-l}$. 

The expansion \eqref{e:full-off} will be called the full off-diagonal expansion for the kernel of $\Xi_p$. 

\begin{defn}\label{d:Toeplitz}
We introduce the class $\mathfrak A$, which consists of sequences of linear operators  $\{T_p: L^2(X,L^p\otimes E)\to L^2(X,L^p\otimes E)\}$, satisfying the following conditions:
\begin{description}
\item[(i)] For any $p\in \mathbb N$, we have
\[
T_p=P_{p,\Lambda}T_pP_{p,\Lambda}. 
\]
\item[(ii)] For any $\varepsilon_0>0$ and $l\in \mathbb N$, there exists $C>0$ such that 
\[
|T_{p}(x,x^\prime)|\leq Cp^{-l}
\]
for any $p\in \mathbb N$ and $(x,x^\prime)\in X\times X$ with $d(x,x^\prime)>\varepsilon_0$. (Here $d(x,x^\prime)$ is the geodesic distance.)
\item[(iii)] 
The  kernel of $T_p$ admits the full off-diagonal expansion 
\[
p^{-n}T_{p,x_0}(Z,Z^\prime)\cong \sum_{r=0}^kK_{r,x_0}(\sqrt{p}Z,\sqrt{p}Z^\prime)p^{-\frac{r}{2}}+\mathcal O(p^{-\frac{k+1}{2}})
\] 
for any $k\in \mathbb N$, $x_0\in X$, $Z,Z^\prime\in T_{x_0}X$, $|Z|, |Z^\prime|<\varepsilon^\prime$ with some $\varepsilon^\prime\in (0,r_X/4)$,
with 
\[
K_{r,x_0}(Z,Z^\prime)=(\mathcal Q_{r,x_0}\mathcal P)(Z,Z^\prime), 
\]
where $\mathcal Q_{r,x_0}\in \operatorname{End}(E_{x_0})[Z,Z^\prime]$ is a family of polynomials, depending smoothly on $x_0$, of the same parity as $r$.   
\end{description}
\end{defn} 

One can easily check the following properties of $\mathfrak A$. 

\begin{prop}\label{p:boundedA}
The set $\mathfrak A$ is an involutive algebra. For any $\{T_p\}\in \mathfrak A$, the operator $T_p$  is bounded in $L^2(X,L^p\otimes E)$ with the norm, uniformly bounded in $p$. 
\end{prop}

For any $F, G\in C^\infty(T_{x_0}X\times T_{x_0}X, \operatorname{End}(E_{x_0}))$, exponentially decreasing away the diagonal, we denote by $F\ast G\in C^\infty(T_{x_0}X\times T_{x_0}X, \operatorname{End}(E_{x_0}))$ the smooth kernel of the composition of the corresponding integral operators in $L^2(T_{x_0}X, E_{x_0})$:
\[
(F\ast G)(Z,Z^\prime)=\int_{T_{x_0}X} F(Z,Z^{\prime\prime})G(Z^{\prime\prime} ,Z^\prime)dZ^{\prime\prime}. 
\]

\begin{prop}\label{p:algebraA}
For any $\{T^\prime_p\}, \{T^{\prime\prime}_p\}\in \mathfrak A$, the coefficients $K_{r,x_0}(Z,Z^\prime)$ in the full off-diagonal expansion for the kernel of the composition $\{T^\prime_p\circ T^{\prime\prime}_p\}$ are related with the analogous coefficients $K^\prime_{r,x_0}(Z,Z^\prime)$ and $K^{\prime\prime}_{r,x_0}(Z,Z^\prime)$ for  $\{T^\prime_p\}$ and $\{T^{\prime\prime}_p\}$, respectively, by 
\begin{equation}\label{e:comp-K}
K_{r,x_0}=\sum_{r_1+r_2=r}K^\prime_{r_1,x_0}\ast K^{\prime\prime}_{r_2,x_0}. 
\end{equation}
\end{prop}

\section{Description of the kernels of Toeplitz operators}\label{s:Toeplitz-in-A}  
In this section, we show that any Toeplitz operator $\{T_{p}\}$ in the sense of Definition \ref{d:Toeplitz0} belongs to the algebra $\mathfrak A$ introduced in the previous section.  It is easy to see that it suffices to do this for the operator $\{T_{f,p}\}$  determined by $f\in C^\infty(X,\operatorname{End}(E))$.

Let $P_{p,\Lambda}(x,x^\prime)$, $x,x^\prime\in X$, be the smooth kernel of $P_{p,\Lambda}$ with respect to the Riemannian volume form $dv_X$. The Schwartz kernel of $T_{f,p}$ is given by  
\begin{equation}\label{e:Tfp}
T_{f,p}(x,x^\prime)=\int_X P_{p,\Lambda}(x,y)f(y)P_{p,\Lambda}(y,x^{\prime})dv_X(y). 
\end{equation}

\begin{lem}
For any $\varepsilon>0$ and $l,m\in \mathbb N$, there exists $C>0$ such that for any $p\geq 1$ and $(x,x^\prime)\in X\times X$ with $d(x,x^\prime)>\varepsilon$ we have
\[
|T_{f,p}(x,x^\prime)|_{C^m}\leq Cp^{-l}. 
\]
\end{lem}

Here $|T_{f,p}(x, x^\prime)|_{\mathscr{C}^k}$ denotes the pointwise $\mathscr{C}^k$-seminorm of the section $T_{f,p}$ at a point $(x, x^\prime)\in X\times X$, which is the sum of the norms induced by $h^L, h^E$ and $g$ of the derivatives up to order $k$ of $T_{f,p}$ with respect to the connection $\nabla^{L^p\otimes E}$ and the Levi-Civita connection $\nabla^{TX}$ evaluated at $(x, x^\prime)$.

\begin{proof}
The proof follows from \eqref{e:Tfp} and the off-diagonal exponential estimate for $P_{p,\Lambda}(x,x^\prime)$ \cite{Kor20} as in \cite[Lemma 4.2]{ma-ma08a}.
\end{proof}

By \cite{Kor20}, for any $k\in \mathbb N$, the kernel of $P_{p,\Lambda}$ admits the full off-diagonal expansion 
\begin{equation}\label{e:PL-exp}
p^{-n}P_{p,\Lambda,x_0}(Z,Z^\prime)\cong
\sum_{r=0}^kF_{r,x_0}(\sqrt{p} Z, \sqrt{p}Z^\prime)p^{-\frac{r}{2}}+\mathcal O(p^{-\frac{k+1}{2}}).
\end{equation}
Here, for any $r\geq 0$, the coefficient $F_{r,x_0}\in C^\infty(T_{x_0}X\times T_{x_0}X, \operatorname{End}(E_{x_0}))$ has the form
\begin{equation}\label{e:Fr}
F_{r,x_0}(Z,Z^\prime)=J_{r,x_0}(Z,Z^\prime)\mathcal P(Z,Z^\prime),
\end{equation}
where $J_{r,x_0}(Z,Z^\prime)$ is a polynomial in $Z, Z^\prime$, depending smoothly on $x_0$, 
of the same parity as $r$ and $\operatorname{deg} J_{r,x_0}\leq \kappa(\Lambda)+3r$, where $\kappa(\Lambda)=\max \{|\mathbf k| : \Lambda_{\mathbf k}=\Lambda \}$. 

Recall the description of the leading coefficient  $F_{0,x_0}(Z,Z^\prime)$.  We introduce the connection on the trivial line bundle on $T_{x_0}X\cong \RR^{2n}$, with the connection one-form $\alpha$ given by 
\begin{equation}\label{e:Aflat}
\alpha_=\sum_{k=1}^{n}\frac{1}{2}a_k(Z_{2k-1}\, dZ_{2k}-Z_{2k}\, dZ_{2k-1}). 
\end{equation}
Its curvature is constant: $d\alpha=\mathbf B_{x_0}$. 

Let $\mathcal H^{(0)}$ be the associated Bochner Laplacian on $C^\infty(\RR^{2n})$:
\begin{equation}\label{e:DeltaL0p}
\mathcal H^{(0)}=(d-i\alpha)^*(d-i\alpha).
\end{equation}
The spectrum of $\mathcal H^{(0)}$ coincides with $\Sigma$ and consists of eigenvalues of infinite multiplicity. Considered as an operator on $C^\infty(T_{x_0}X,E_{x_0})\cong C^\infty(\RR^{2n},E_{x_0})$, the operator $\mathcal H^{(0)}\otimes \operatorname{id}_{E_{x_0}}$ is exactly the model operator at $x_0$ mentioned in Introduction. So the assumption \eqref{e:aj-constant} guarantees that, in a suitable coordinates, the model operator is independent of $x_0$. Let $\mathcal P_{\Lambda}$ be the orthogonal projection on the eigenspace of $\mathcal H^{(0)}$ with the eigenvalue $\Lambda$ (see Section \ref{s:model} for more information on $\mathcal P_{\Lambda}$).

The leading coefficient in \eqref{e:PL-exp} is given by 
\begin{equation}\label{e:F0}
F_{0,x_0}(Z,Z^\prime)=\mathcal P_{\Lambda}(Z,Z^\prime).
\end{equation}
As in \cite[Lemma 4.7]{ma-ma08a}, using an explicit formula for $F_{1,x_0}$ given in \cite{Kor20}, one can show that, for any $Z, Z^\prime\in T_{x_0}X$, $F_{1,x_0}(Z,Z^\prime)$ is a scalar operator in $E_{x_0}$, and, therefore, commutes with any operator in $E_{x_0}$.

Since $P_{p,\Lambda}$ is a projection, we have 
\begin{gather}
\mathcal P_{\Lambda}\ast F_{1,x_0}+F_{1,x_0} \ast \mathcal P_{\Lambda}=F_{1,x_0},\label{e:F1} \\
\mathcal P_{\Lambda}\ast F_{2,x_0}+F_{1,x_0}\ast F_{1,x_0}+ F_{2,x_0} \ast \mathcal P_{\Lambda}=F_{2,x_0}. \label{e:F2}
\end{gather}

\begin{lem}
For any $f\in C^\infty(X, \operatorname{End}(E))$, the operator $\{T_{f,p}\}$ belongs to $\mathfrak A$. The coefficients $K_{r,x_0}(f)\in C^\infty(T_{x_0}X\times T_{x_0}X, \operatorname{End}(E_{x_0}))$ of the full off-diagonal expansion for the kernel of $T_{f,p}$ are given by 
\begin{equation}\label{e:Krf}
K_{r,x_0}(f)=\sum_{r_1+r_2+k=r}F_{r_1,x_0}\ast \frac{1}{k!}(d^kf_{x_0})_0\cdot F_{r_2,x_0},
\end{equation}
where 
\[
(d^kf_{x_0})_0(Z)=\sum_{|\alpha|=k}\frac{\partial^\alpha f_{x_0}}{\partial Z^\alpha}(0)Z^\alpha
\]
and we denote 
\[
(d^kf_{x_0})_0\cdot F_{r_2,x_0}(Z,Z^\prime)=(d^kf_{x_0})_0(Z) F_{r_2,x_0}(Z,Z^\prime)
\]
In particular, we have
\begin{align}
K_{0,x_0}(f)=&f(x_0)\mathcal P_{\Lambda},\label{e:K0f}\\
K_{1,x_0}(f)=& f(x_0)F_{1,x_0}+\mathcal P_{\Lambda}\ast (df_{x_0})_0\cdot \mathcal P_{\Lambda},\label{e:K1f} 
\end{align}
and, for $f\in C^\infty(X)$
\begin{multline}
K_{2,x_0}(f)= f(x_0)F_{2,x_0}+F_{1,x_0}\ast (df_{x_0})_0\cdot \mathcal P_{\Lambda} \\ + \mathcal P_{\Lambda}\ast (df_{x_0})_0\cdot F_{1,x_0}+ \mathcal P_{\Lambda}\ast \frac 12 (d^2f_{x_0})_0\cdot \mathcal P_{\Lambda}. \label{e:K2f}
\end{multline}
\end{lem} 
 
\begin{proof}
The proof goes along the same lines as the proof of \cite[Lemma 4.6]{ma-ma08a}, so we just highlight the main points.

The fact that $\{T_{f,p}\}$ belongs to $\mathfrak A$ follows easily from \eqref{e:Tfp}, \eqref{e:PL-exp} and Proposition \ref{p:algebraA}.
By \eqref{e:Tfp} and \eqref{e:PL-exp}, we also get
\begin{multline*}
p^{-n}T_{f,p,x_0}(Z,Z^\prime)\cong \sum_{r=0}^\infty p^{-\frac{r}{2}} \sum_{r_1+r_2=r} \int F_{r_1,x_0}(\sqrt{p} Z, \sqrt{p}W)f_{x_0}(W)\times \\ \times  F_{r_2,x_0}(\sqrt{p} W, \sqrt{p}Z^\prime)dW.
\end{multline*}
Now we write the Taylor expansion for $f_{x_0}$ at $0$:
\[
f_{x_0}(W) =\sum_{\alpha\in \ZZ_+^{2n}} \frac{\partial^\alpha f_{x_0}}{\partial W^\alpha}(0)\frac{W^\alpha}{\alpha!},
\]
We infer that
\begin{multline*}
p^{-n}T_{f,p,x_0}(Z,Z^\prime)\cong \sum_{\alpha\in \ZZ_+^{2n}} \sum_{r=0}^\infty p^{-\frac{r}{2}} \sum_{r_1+r_2+|\alpha|=r} \int F_{r_1,x_0}(\sqrt{p} Z, \sqrt{p}Z^\prime)\times \\ \times \frac{\partial^\alpha f_{x_0}}{\partial W^\alpha}(0)\frac{(\sqrt{p} W)^\alpha}{\alpha!}   F_{r_2,x_0}(\sqrt{p} W, \sqrt{p}Z^\prime)dW,
\end{multline*} 
which proves \eqref{e:Krf}.
 
Using \eqref{e:Krf}, \eqref{e:F1}, \eqref{e:F2} and the fact that $F_{1,x_0}(Z,Z^\prime)$ commutes with any operator in $E_{x_0}$, one can easily derive \eqref{e:K0f}, \eqref{e:K1f} and \eqref{e:K2f}.
\end{proof}

\section{The composition theorem}\label{s:Thm1}
Now we will use the results of the previous sections to prove Theorem \ref{t:comm}. In this section, we will prove the first part of the theorem and reduce the proof of the second part to the proof of a similar statement in the Euclidean case, which will be given in the next section.   

Let $f,g\in C^\infty(X, \operatorname{End}(E))$. By Proposition \ref{p:algebraA}, the operator $T_{f,p}T_{g,p}$ belongs to $\mathfrak A$, and, by \eqref{e:comp-K}, the coefficients $K_{r,x_0}(f,g)\in C^\infty(T_{x_0}X\times T_{x_0}X, \operatorname{End}(E_{x_0}))$ of the full off-diagonal expansion for the kernel of $T_{f,p}T_{g,p}$ are given by 
\begin{equation}\label{e:Krf-g}
K_{r,x_0}(f,g)=\sum_{r_1+r_2=r}K_{r_1,x_0}(f)\ast K_{r_2,x_0}(g) 
\end{equation}
To prove \eqref{e:TfpTgp-prod}, it is sufficient to show that
\[
K_{r,x_0}(f,g)=K_{r,x_0}(fg). \quad r=0,1.
\]
By \eqref{e:Krf-g}, \eqref{e:K0f} and the fact that  $\mathcal P_{\Lambda}$ commutes with $g(x_0)$, we get
\begin{align*}
K_{0,x_0}(f,g)=& K_{0,x_0}(f)\ast K_{0,x_0}(g)\\ =& f(x_0)\mathcal P_{\Lambda}\ast g(x_0)\mathcal P_{\Lambda}=f(x_0)g(x_0)\mathcal P_{\Lambda}=K_{0,x_0}(fg). 
\end{align*}
Next, by \eqref{e:Krf-g}, \eqref{e:K0f}, \eqref{e:K1f}, \eqref{e:F1}, we have  
\begin{align*}
K_{1,x_0}(f,g)=& K_{1,x_0}(f)\ast K_{0,x_0}(g)+K_{0,x_0}(f)\ast K_{1,x_0}(g)\\ 
=& (f(x_0)F_{1,x_0}+\mathcal P_{\Lambda}\ast (df_{x_0})_0\cdot \mathcal P_{\Lambda}) \ast g(x_0)\mathcal P_{\Lambda}\\ &+f(x_0)\mathcal P_{\Lambda}\ast (g(x_0)F_{1,x_0}+\mathcal P_{\Lambda}\ast (dg_{x_0})_0\cdot \mathcal P_{\Lambda}) \\ 
=& f(x_0)g(x_0)F_{1,x_0}+\mathcal P_{\Lambda}\ast (df_{x_0})_0 g(x_0)\cdot \mathcal P_{\Lambda}\\ &+\mathcal P_{\Lambda}\ast f(x_0)(dg_{x_0})_0\cdot \mathcal P_{\Lambda}\\ 
=& f(x_0)g(x_0)F_{1,x_0}+\mathcal P_{\Lambda}\ast (d(fg)_{x_0})_0 \cdot \mathcal P_{\Lambda}\\ = & K_{1,x_0}(fg), 
\end{align*}
since  $\mathcal P_{\Lambda}$ and $F_{1,x_0}$ commute with $g(x_0)$ and $f(x_0)$ and $h\cdot \mathcal P_{\Lambda} \ast \mathcal P_{\Lambda}=h\cdot \mathcal P_{\Lambda}$ for any $h$.
This proves \eqref{e:TfpTgp-prod}.

Now suppose that $f,g\in C^\infty(X)$. Then we have 
\begin{equation}\label{e:K01fg-gf}
K_{r,x_0}(f,g)-K_{r,x_0}(g,f)=0, \quad r=0,1.
\end{equation}
Let us compute $K_{2,x_0}(f,g)$. By \eqref{e:Krf-g}, \eqref{e:K0f}, \eqref{e:K1f}, \eqref{e:K2f},  we have
\begin{align*}
K_{2,x_0}(f,g)
=& \Big(f(x_0)F_{2,x_0}+F_{1,x_0}\ast (df_{x_0})_0\cdot \mathcal P_{\Lambda} + \mathcal P_{\Lambda}\ast (df_{x_0})_0\cdot F_{1,x_0} \\ & + \mathcal P_{\Lambda}\ast \frac 12 (d^2f_{x_0})_0\cdot \mathcal P_{\Lambda}\Big)\ast g(x_0)\mathcal P_{\Lambda}\\
 & +(f(x_0)F_{1,x_0}+\mathcal P_{\Lambda}\ast (df_{x_0})_0\cdot \mathcal P_{\Lambda})\ast (g(x_0)F_{1,x_0}+\mathcal P_{\Lambda}\ast (dg_{x_0})_0\cdot \mathcal P_{\Lambda})\\ & + f(x_0)\mathcal P_{\Lambda} \ast \Big(g(x_0)F_{2,x_0}+F_{1,x_0}\ast (dg_{x_0})_0\cdot \mathcal P_{\Lambda}\\ & + \mathcal P_{\Lambda}\ast (dg_{x_0})_0\cdot F_{1,x_0}+ \mathcal P_{\Lambda}\ast \frac 12 (d^2g_{x_0})_0\cdot \mathcal P_{\Lambda}\Big). 
\end{align*}
Using \eqref{e:F2}, we collect the terms with $f(x_0)g(x_0)$. Since $\mathcal P_{\Lambda}$ and $F_{1,x_0}$ commute with $g(x_0)$ and $f(x_0)$, $h\cdot \mathcal P_{\Lambda} \ast \mathcal P_{\Lambda}=h\cdot \mathcal P_{\Lambda}$ for any $h$, we get
\begin{align*}
K_{2,x_0}(f,g)=& f(x_0)g(x_0)F_{2,x_0}\\ & +F_{1,x_0}\ast (df_{x_0})_0g(x_0)\cdot \mathcal P_{\Lambda} + \mathcal P_{\Lambda}\ast (df_{x_0})_0g(x_0)\cdot F_{1,x_0}\ast \mathcal P_{\Lambda} \\ & + \mathcal P_{\Lambda}\ast \frac 12 (d^2f_{x_0})_0g(x_0) \cdot \mathcal P_{\Lambda}+F_{1,x_0}\ast f(x_0)(dg_{x_0})_0\cdot \mathcal P_{\Lambda}\\ & +\mathcal P_{\Lambda}\ast (df_{x_0})_0g(x_0)\cdot \mathcal P_{\Lambda}\ast F_{1,x_0}+\mathcal P_{\Lambda}\ast (df_{x_0})_0\cdot \mathcal P_{\Lambda}\ast (dg_{x_0})_0\cdot \mathcal P_{\Lambda}\\  & + \mathcal P_{\Lambda} \ast f(x_0) (dg_{x_0})_0\cdot F_{1,x_0}
+ \mathcal P_{\Lambda}\ast \frac 12 f(x_0)(d^2g_{x_0})_0\cdot \mathcal P_{\Lambda}\\ =& f(x_0)g(x_0)F_{2,x_0}+F_{1,x_0}\ast (d(fg)_{x_0})_0\cdot \mathcal P_{\Lambda}+ \mathcal P_{\Lambda}\ast  (d(fg)_{x_0})_0\cdot F_{1,x_0} \\ & + \mathcal P_{\Lambda}\ast \frac 12 (d^2f_{x_0})_0g(x_0) \cdot \mathcal P_{\Lambda} +\mathcal P_{\Lambda}\ast (df_{x_0})_0\cdot \mathcal P_{\Lambda}\ast (dg_{x_0})_0\cdot \mathcal P_{\Lambda}\\ & + \mathcal P_{\Lambda}\ast \frac 12 f(x_0)(d^2g_{x_0})_0\cdot \mathcal P_{\Lambda}. 
\end{align*}
Finally, we see that 
\begin{multline}\label{e:K2fg-gf}
K_{2,x_0}(f,g)-K_{2,x_0}(g,f)\\= \mathcal P_{\Lambda}\ast (df_{x_0})_0\cdot \mathcal P_{\Lambda}\ast (dg_{x_0})_0\cdot \mathcal P_{\Lambda}-\mathcal P_{\Lambda}\ast (dg_{x_0})_0\cdot \mathcal P_{\Lambda}\ast (df_{x_0})_0\cdot \mathcal P_{\Lambda}.
\end{multline}

Thus, we have reduced the proof of \eqref{e:TfpTgp-comm} to the linear model. 

We consider the linear space $\mathbb R^{2n}$ equipped with the symplectic form $\omega_a=\sum_{k=1}^{n}a_kdZ_{2k-1}\wedge dZ_{2k}$ (cf. \eqref{e:Bx0}).

\begin{prop}\label{p:linear-comm}
For linear functions $F$ and $G$ on $\mathbb R^{2n}$, we have
\[
[P_{\Lambda} F \mathcal P_{\Lambda}, P_{\Lambda} G \mathcal P_{\Lambda}]=\{F,G\}_a \mathcal P_{\Lambda},
\]
where $\{F,G\}_a$ is the Poisson bracket on the symplectic manifold $(\mathbb R^{2n},\omega_a)$
\end{prop}

Proof of Proposition \ref{p:linear-comm} will be given in the next section. Now we demonstrate how it allows us to complete the proof of  \eqref{e:TfpTgp-comm}.

Observe that, for linear functions $F$ and $G$, the Poisson bracket $\{F,G\}_a$ is a constant function on $\RR^{2n}$, and, by \eqref{e:Bx0}, it is easy to see that, for any $f,g\in C^\infty(X)$, 
\[
\{f,g\}(x_0)=\{(df_{x_0})_0,(dg_{x_0})_0\}_a.
\]
Therefore, the identity \eqref{e:TfpTgp-comm} follows immediately from \eqref{e:K01fg-gf}, \eqref{e:K2fg-gf} and Propositions \ref{p:linear-comm} and \ref{p:boundedA}.

\section{The model case}\label{s:model}

In this section, we prove Proposition \ref{p:linear-comm}, thus completing the proof Theorem \ref{t:comm}. First, we need to recall some information on the spectral theory of the model operator $\mathcal H^{(0)}$ \cite{ma-ma08a}. 

We will use the complex coordinates $z\in\mathbb C^{n}\cong \mathbb R^{2n}$, $z_j=Z_{2j-1}+iZ_{2j}, j=1,\ldots,n,$ in the linear space $\mathbb R^{2n}$. Put
\[
\frac{\partial}{\partial z_j}=\tfrac{1}{2}\left(\frac{\partial}{\partial Z_{2j-1}}-i\frac{\partial}{\partial Z_{2j}}\right), \quad \frac{\partial}{\partial \bar{z}_j}=\frac{1}{2}\left(\tfrac{\partial}{\partial Z_{2j-1}}+i\frac{\partial}{\partial Z_{2j}}\right).
\]
Define first order differential operators $b_j,b^{+}_j, j=1,\ldots,n,$ on $C^\infty(\RR^{2n},E_{x_0})$ by the formulas
\[
b_j= -2{\frac{\partial}{\partial z_j}}+\frac{1}{2}a_j\bar{z}_j,\quad
b^{+}_j=2{\frac{\partial}{\partial\bar{z}_j}}+\frac{1}{2}a_j z_j, \quad j=1,\ldots,n.
\]
Then $b^{+}_j$ is the formal adjoint of $b_j$ on $L^2(\RR^{2n},E_{x_0})$, and
\[
\mathcal H^{(0)}=\sum_{j=1}^n b_j b^{+}_j+\Lambda_0. 
\]
We have the commutation relations
\begin{equation}
[b_i,b^{+}_j]=b_i b^{+}_j-b^{+}_j b_i =-2a_i \delta_{i\,j},\quad 
[b_i,b_j]=[b^{+}_i,b^{+}_j]=0\, ,\label{e:com1}
\end{equation} 
and, for any polynomial  $g(z,\bar{z})$ on $z$ and $\bar{z}$, 
\begin{equation}
[g(z,\bar{z}),b_j]=  2 \tfrac{\partial}{\partial z_j}g(z,\bar{z}), 
\quad  [g(z,\bar{z}),b_j^+]
= - 2\tfrac{\partial}{\partial \bar{z}_j}g(z,\bar{z})\,. \label{com-bg}
\end{equation}

By \cite[(1.98)]{ma-ma08}, we have
\begin{equation}\label{bP}
(b^{+}_j\mathcal P)(Z,Z^\prime)=0, \quad (b_j\mathcal P)(Z,Z^\prime)=a_j(\bar z_j-\bar z^\prime_j)\mathcal P(Z,Z^\prime).
\end{equation}

We also introduce first order differential operators $\bar b_j, \bar b^{+}_j, j=1,\ldots,n,$ on $C^\infty(\RR^{2n},E_{x_0})$ by the formulas
\[
\bar b_j= -2{\tfrac{\partial}{\partial \bar z_j}}+\tfrac{1}{2}a_j{z}_j,\quad
\bar b^{+}_j=2{\tfrac{\partial}{\partial {z}_j}}+\tfrac{1}{2}a_j \bar z_j, \quad j=1,\ldots,n.
\]
They commute with the operators $b_j,b^{+}_j, j=1,\ldots,n,$ and satisfy the same commutation relations \eqref{e:com1} as $b_j,b^{+}_j$.  
We have
\begin{equation}\label{bar-bP}
(\bar b^{+}_j\mathcal P)(Z,Z^\prime)=a_j\bar z^\prime_j\mathcal P(Z,Z^\prime), \quad (\bar b_j\mathcal P)(Z,Z^\prime)=a_jz_j\mathcal P(Z,Z^\prime).
\end{equation}

Recall that any function $\Phi\in L^2(\RR^{2n},E_{x_0})$ of the form 
\begin{equation}\label{e:eigenspace}
\Phi=b^{\mathbf k}\left(f(z)\exp\left({-\frac{1}{4}\sum_{j=1}^n a_j|z_j|^2}\right)\right), 
\end{equation}
where $f$ is an analytic function in $\CC^n\cong \RR^{2n}$ and $\mathbf k\in\ZZ_+^n$, is an eigenfunction of the operator $\mathcal H^{(0)}$ with the eigenvalue $\Lambda_{\mathbf k}=\sum_{j=1}^n(2k_j+1) a_j$. 
In particular, the eigenspace of $\mathcal H^{(0)}$ associated with an eigenvalue $\Lambda$ consists of functions $\Phi$ given by \eqref{e:eigenspace} with $\mathbf k\in\mathcal K_\Lambda$.

In the case when $E_{x_0}=\CC$, an orthonormal basis of the eigenspace of $\mathcal H^{(0)}$ associated with the lowest eigenvalue $\Lambda_0$ is formed by the functions
\begin{equation}\label{e:varphi-beta}
\varphi_\beta(Z)=\left(\frac{a ^\beta}{(2\pi)^n 2 ^{|\beta|} \beta!}
\prod_{i=1}^n a_i\right)^{1/2}z^\beta
\exp\Big (-\frac{1}{4} \sum_{j=1}^n a_j |z_j|^2\Big )\,,\quad \beta\in\ZZ_+^n.
\end{equation}
An orthonormal basis of the eigenspace associated with the eigenvalue $\Lambda$ is given by (see, for instance, \cite{BPR04})
\begin{equation}\label{e:varphi-kbeta}
\varphi_{\mathbf k,\beta}=\frac{1}{(2^{|\mathbf k|}a^{\mathbf k}\mathbf k!)^{1/2}}b^{\mathbf k}\varphi_\beta\,,\quad \beta\in\ZZ_+^n, \mathbf k\in\mathcal K_\Lambda.
\end{equation}
Thus, the spectral projection $\mathcal P_{\Lambda}$ in $L^2(\RR^{2n},E_{x_0})$ is given by 
\[
\mathcal P_{\Lambda}=\sum_{\mathbf k\in \mathcal K_\Lambda} \mathcal P_{\Lambda_{\mathbf k}},
\]
where $\mathcal P_{\Lambda_{\mathbf k}}$ is the smoothing operator with the kernel 
\begin{equation}\label{e:PLambda-kernel}
\mathcal P_{\Lambda_{\mathbf k}}(Z,Z^\prime)=\sum_{\beta\in\ZZ_+^n}\varphi_{\mathbf k,\beta}(Z) \overline{\varphi_{\mathbf k,\beta}(Z^\prime)}=\frac{1}{2^{|\mathbf k|}a^{\mathbf k}\mathbf k!}b^{\mathbf k}_z\bar{b}^{\mathbf k}_{z^\prime}\mathcal P(Z,Z^\prime).
\end{equation}
(see also \eqref{e:PLambda-k} below for an explicit formula).
We can also write the following formula for the operator $\mathcal P_{\Lambda_{\mathbf k}}$ itself:
\begin{equation}\label{e:PLambda}
\mathcal P_{\Lambda_{\mathbf k}}=\frac{1}{2^{|{\mathbf k}|}a^{\mathbf k}{\mathbf k}!}b^{\mathbf k}\mathcal P(b^+)^{\mathbf k}.
\end{equation}

Observe that
\[
b^+_j \mathcal P_{\Lambda_{\mathbf k}}=\mathcal P_{\Lambda_{{\mathbf k}-e_j}}b^+_j , \quad b_j\mathcal P_{\Lambda_{\mathbf k}}= \mathcal P_{\Lambda_{{\mathbf k}+e_j}}b_j, \quad j=1,\ldots,n,
\]
where $(e_1,\ldots, e_n)$ is the standard basis in $\ZZ^n$. 

Indeed, using \eqref{e:PLambda} and \eqref{e:com1},  we get
\[
b^+_j \mathcal P_{\Lambda_{\mathbf k}}=\frac{1}{2^{|{\mathbf k}|}a^{\mathbf k}{\mathbf k}!}[b^+_j,b^{\mathbf k}] \mathcal P (b^+)^{\mathbf k}=\frac{1}{2^{|{\mathbf k}|}a^{\mathbf k}{\mathbf k}!}2a_jk_j b^{{\mathbf k}-e_j} \mathcal P (b^+)^{\mathbf k}= \mathcal P_{\Lambda_{{\mathbf k}-e_j}}b^+_j .
\]
The second identity follows by taking adjoints.

Next, we show that, for a linear function $F(z,\bar z)=\sum_{j=1}^nF_{z_j}z_j+F_{\bar z_j}\bar z_j$, we have
\begin{multline}\label{e:commFP}
[F,\mathcal P_{\Lambda_{\mathbf k}}]\\ = \sum_{j=1}^n \frac{1}{a_j} \left[ F_{z_j}b^+_j \left(\mathcal P_{\Lambda_{\mathbf k}} - \mathcal P_{\Lambda_{{\mathbf k}+e_j}}\right)+ F_{\bar z_j} b_j  \left(\mathcal P_{\Lambda_{\mathbf k}}- \mathcal P_{\Lambda_{{\mathbf k}-e_j}} \right)\right]. 
\end{multline} 
Observe that
\begin{equation}\label{e:commzP}
[z_j,\mathcal P_{\Lambda_{\mathbf k}}] = \frac{1}{a_j}\left(\mathcal P_{\Lambda_{{\mathbf k}-e_j}} - \mathcal P_{\Lambda_{\mathbf k}}\right) b^+_j= \frac{1}{a_j} b^+_j \left(\mathcal P_{\Lambda_{\mathbf k}} - \mathcal P_{\Lambda_{{\mathbf k}+e_j}}\right). 
\end{equation} 
Indeed,  using \eqref{e:PLambda}, \eqref{com-bg} and \eqref{bP},  we get
\begin{align*}
[z_j,\mathcal P_{\Lambda_{\mathbf k}}]= & \frac{1}{2^{|{\mathbf k}|}a^{\mathbf k}{\mathbf k}!}\left([z_j,b^{\mathbf k}] \mathcal P (b^+)^{\mathbf k}+b^{\mathbf k} [z_j,\mathcal P] (b^+)^{\mathbf k}\right)\\ = & \frac{1}{2^{|{\mathbf k}|}a^{\mathbf k}{\mathbf k}!}\left(2k_jb^{{\mathbf k}-e_j} \mathcal P (b^+)^{\mathbf k}-\frac{1}{a_j}b^{\mathbf k} \mathcal P (b^+)^{{\mathbf k}+e_j}\right)\\ = & \frac{1}{a_j}\left(\mathcal P_{\Lambda_{{\mathbf k}-e_j}} - \mathcal P_{\Lambda_{\mathbf k}}\right) b^+_j. 
\end{align*}
Taking adjoints, we infer that
\begin{equation}\label{e:commzP2}
[\bar z_j,\mathcal P_{\Lambda_{\mathbf k}}]= \frac{1}{a_j} b_j  \left(\mathcal P_{\Lambda_{\mathbf k}}- \mathcal P_{\Lambda_{{\mathbf k}-e_j}} \right)= \frac{1}{a_j}   \left(\mathcal P_{\Lambda_{{\mathbf k}+e_j}}- \mathcal P_{\Lambda_{\mathbf k}} \right)b_j. 
\end{equation} 
From \eqref{e:commzP} and \eqref{e:commzP2}, we get \eqref{e:commFP}.


Using \eqref{e:commFP}, we compute
\begin{multline}\label{e:PFP}
P_{\Lambda_{{\mathbf k}_1}} F \mathcal P_{\Lambda_{{\mathbf k}_2}}=\delta_{{\mathbf k}_1,{\mathbf k}_2} F \mathcal P_{\Lambda_{{\mathbf k}_2}}+[P_{\Lambda_{{\mathbf k}_1}}, F] \mathcal P_{\Lambda_{{\mathbf k}_2}}
= \delta_{{\mathbf k}_1,{\mathbf k}_2} F \mathcal P_{\Lambda_{{\mathbf k}_2}}\\ +\sum_{\ell=1}^n \frac{1}{a_\ell}\left[ \left(\delta_{{\mathbf k}_1-e_\ell,{\mathbf k}_2}- \delta_{{\mathbf k}_1,{\mathbf k}_2} \right)F_{\bar z_\ell} b_\ell -\left(\delta_{{\mathbf k}_1,{\mathbf k}_2} - \delta_{{\mathbf k}_1+e_\ell,{\mathbf k}_2}\right) F_{z_\ell} b^+_\ell \right]\mathcal P_{\Lambda_{{\mathbf k}_2}} .
\end{multline}

Now we are ready to  complete the proof of Proposition~\ref{p:linear-comm}. For linear functions $F$ and $G$, using \eqref{e:PFP}, we get 
\begin{multline*}
P_{\Lambda_{{\mathbf k}_1}} F \mathcal P_{\Lambda_{\mathbf k}}G\mathcal P_{\Lambda_{{\mathbf k}_2}} = \delta_{{\mathbf k}_1,{\mathbf k}} F \mathcal P_{\Lambda_{\mathbf k}}G\mathcal P_{\Lambda_{{\mathbf k}_2}} \\ +\sum_{\ell=1}^n\frac{1}{a_\ell}\left[ \left(\delta_{{\mathbf k}_1-e_\ell,{\mathbf k}}- \delta_{{\mathbf k}_1,{\mathbf k}} \right)F_{\bar z_\ell} b_\ell -\left(\delta_{{\mathbf k}_1,{\mathbf k}} - \delta_{{\mathbf k}_1+e_\ell,{\mathbf k}}\right) F_{z_\ell} b^+_\ell \right]\mathcal P_{\Lambda_{\mathbf k}} G\mathcal P_{\Lambda_{{\mathbf k}_2}}.
\end{multline*}
Next, we transpose $\mathcal P_{\Lambda_{\mathbf k}}$ and $G$ and apply \eqref{e:commFP}:
\begin{multline*}
P_{\Lambda_{{\mathbf k}_1}} F \mathcal P_{\Lambda_{\mathbf k}}G\mathcal P_{\Lambda_{{\mathbf k}_2}} = \delta_{{\mathbf k}_1,{\mathbf k}}\delta_{{\mathbf k},{\mathbf k}_2}  FG \mathcal P_{\Lambda_{{\mathbf k}_2}} \\ 
\begin{aligned}
& +\sum_{\ell=1}^n \frac{1}{a_\ell} \delta_{{\mathbf k},{\mathbf k}_2}\left[ \left(\delta_{{\mathbf k}_1-e_\ell,{\mathbf k}}- \delta_{{\mathbf k}_1,{\mathbf k}} \right) F_{\bar z_\ell}b_\ell  -\left(\delta_{{\mathbf k}_1,{\mathbf k}} - \delta_{{\mathbf k}_1+e_\ell,{\mathbf k}}\right) F_{z_\ell} b^+_\ell \right] G \mathcal P_{\Lambda_{{\mathbf k}_2}} \\ & + \sum_{j=1}^n \frac{1}{a_j} \delta_{{\mathbf k}_1,{\mathbf k}} F \left[\left(\delta_{{\mathbf k}+e_j,{\mathbf k}_2}-\delta_{{\mathbf k},{\mathbf k}_2}\right) G_{z_j}b^+_j + \left(\delta_{{\mathbf k}-e_j,{\mathbf k}_2}-\delta_{{\mathbf k},{\mathbf k}_2}\right)  G_{\bar z_j} b_j   \right] \mathcal P_{\Lambda_{{\mathbf k}_2}} \\  & +\sum_{j,\ell=1}^n \frac{1}{a_ja_\ell}\left[ \left(\delta_{{\mathbf k}_1-e_\ell,{\mathbf k}}- \delta_{{\mathbf k}_1,{\mathbf k}} \right)F_{\bar z_\ell}b_\ell -\left(\delta_{{\mathbf k}_1,{\mathbf k}} - \delta_{{\mathbf k}_1+e_\ell,{\mathbf k}}\right) F_{z_\ell} b^+_\ell\right]\times \\ & \times  \left[\left(\delta_{{\mathbf k}+e_j,{\mathbf k}_2}-\delta_{{\mathbf k},{\mathbf k}_2}\right) G_{z_j}b^+_j + \left(\delta_{{\mathbf k}-e_j,{\mathbf k}_2}-\delta_{{\mathbf k},{\mathbf k}_2}\right) G_{\bar z_j} b_j \right] \mathcal P_{\Lambda_{{\mathbf k}_2}}.
\end{aligned}
\end{multline*}
In the case when ${\mathbf k}_1,{\mathbf k},{\mathbf k}_2\in \mathcal K_\Lambda$, we necessarily have ${\mathbf k}_1\pm e_\ell\not\in \mathcal K_\Lambda$, ${\mathbf k}\pm e_j\not\in \mathcal K_\Lambda$. Thus, from the last formula, we conclude that $P_{\Lambda_{{\mathbf k}_1}} F \mathcal P_{\Lambda_{\mathbf k}}G\mathcal P_{\Lambda_{{\mathbf k}_2}} =0$, unless ${\mathbf k}_1={\mathbf k}_2={\mathbf k}$, and in the latter case, we have
 \begin{multline*}
P_{\Lambda_{\mathbf k}} F \mathcal P_{\Lambda_{\mathbf k}}G\mathcal P_{\Lambda_{\mathbf k}} = FG \mathcal P_{\Lambda_{\mathbf k}} \\ 
+\sum_{\ell=1}^n \frac{1}{a_\ell}\left[-F_{\bar z_\ell}b_\ell -F_{z_\ell} b^+_\ell \right] G \mathcal P_{\Lambda_{\mathbf k}} + \sum_{j=1}^n \frac{1}{a_j} F \left[-G_{z_j}b^+_j-G_{\bar z_j} b_j   \right] \mathcal P_{\Lambda_{\mathbf k}} \\  +\sum_{j,\ell=1}^n \frac{1}{a_ja_\ell}\left[ - F_{\bar z_\ell}b_\ell - F_{z_\ell} b^+_\ell\right] \left[- G_{z_j}b^+_j - G_{\bar z_j} b_j \right] \mathcal P_{\Lambda_{\mathbf k}}.
\end{multline*}
We see that 
\begin{multline*}
P_{\Lambda_{\mathbf k}} F \mathcal P_{\Lambda_{\mathbf k}}G\mathcal P_{\Lambda_{\mathbf k}}-P_{\Lambda_{\mathbf k}} G \mathcal P_{\Lambda_{\mathbf k}}F\mathcal P_{\Lambda_{\mathbf k}}  =\sum_{j=1}^n \frac{1}{a^2_j}(F_{z_j} G_{\bar z_j}-F_{\bar z_j}G_{z_j})(b^+_j b_j - b_j  b^+_j) \mathcal P_{\Lambda_{\mathbf k}}\\ =\sum_{j=1}^n \frac{2}{a_j}(F_{z_j} G_{\bar z_j}-F_{\bar z_j}G_{z_j}) \mathcal P_{\Lambda_{\mathbf k}}=\{F,G\}_a\mathcal P_{\Lambda_{\mathbf k}}.
\end{multline*}
that completes the proof of Proposition \ref{p:linear-comm}.

\section{Characterization of Toeplitz operators}\label{s:charact}

In the case when $\mathcal K_\Lambda$ consists of a single element, we prove that the sect of Toeplitz operators coincides with the algebra $\mathfrak A$, which gives a characterization of Toeplitz operators in terms of their Schwartz kernels, This type of characterization was introduced in \cite[Theorem 4.9]{ma-ma08}. 
Using this result and Theorems~\ref{t:comm},  we easily complete the proof of Theorem \ref{t:algebra}. 

\begin{thm}\label{t:characterization}
Assume that $\mathcal K_\Lambda$ consists of a single element.
A sequence of bounded linear operators $\{T_p: L^2(X,L^p\otimes E)\to L^2(X,L^p\otimes E)\}$ is a Toeplitz operator in the sense of Definition~\ref{d:Toeplitz0} if and only if it belongs to $\mathfrak A$.
\end{thm} 

The fact that any Toeplitz operator in the sense of Definition~\ref{d:Toeplitz0} belongs to $\mathfrak A$ is proved in Section~\ref{s:Toeplitz-in-A} and holds without any assumption on $\Lambda$. 
Thus, we assume that $\{T_p\}$ belongs to $\mathfrak A$ and prove that it  is a Toeplitz operator in the sense of Definition~\ref{d:Toeplitz0}. The proof is divided in several steps and in the beginning we don't assume that $\mathcal K_\Lambda$ consists of a single element.

The following is an analog of \cite[Lemma 4.12]{ma-ma08a}.  Recall that $K_{0,x_0}(Z,Z^\prime)$ denotes the leading coefficient in the full off-diagonal expansion for the kernel of $T_p$, 

\begin{prop}\label{p:K_0zz}
The  coefficient $K_{0,x_0}(Z,Z^\prime)$ has the form
\[
K_{0,x_0}(Z,Z^\prime)=\sum_{{\mathbf k},{\mathbf k}^\prime\in\cK_\Lambda}b^{\mathbf k}_z{\bar b}^{{\mathbf k}^\prime}_{z^\prime}[Q_{{\mathbf k}{\mathbf k}^\prime,x_0}\mathcal P] 
\]
with some polynomials $Q_{{\mathbf k}{\mathbf k}^\prime,x_0}(z,\bar z^\prime)$
 for any $x_0\in X$ and $Z,Z^\prime\in T_{x_0}X$. 
\end{prop}

\begin{proof}
By Definition~\ref{d:Toeplitz} (i) and \eqref{e:comp-K}, we get
\begin{equation}\label{e:K=RKP}
K_{0,x_0}=\mathcal P_{\Lambda}\circ K_{0,x_0} \circ \mathcal P_{\Lambda}
\end{equation}
By \eqref{e:K=RKP} and \eqref{e:PLambda}, it follows that 
\begin{equation}\label{e:K}
K_{0,x_0}=\mathcal P_{\Lambda}\circ K_{0,x_0} =\sum_{{\mathbf k}\in\cK_\Lambda}\frac{1}{2^{|{\mathbf k}|}a^{\mathbf k}{\mathbf k}!}b^{\mathbf k}\circ \mathcal P\circ (b^+)^{\mathbf k}\circ K_{0,x_0}.
\end{equation}
By \cite[(2.12)]{ma-ma08a}, there exists $F_k\in \CC[z, Z^\prime]$ such that 
\[
(\mathcal P\circ (b^+)^{\mathbf k}\circ K_{0,x_0})(Z, Z^\prime)=F_{\mathbf k}\cdot \cP(Z, Z^\prime). 
\]
Plugging this in \eqref{e:K}, we get
\begin{equation}\label{e:K2}
K_{0,x_0}(Z, Z^\prime)=\sum_{{\mathbf k}\in\cK_\Lambda}\frac{1}{2^{|{\mathbf k}|}a^{\mathbf k}{\mathbf k}!}b^{\mathbf k}_zF_{\mathbf k}(z, Z^\prime)\cP(Z, Z^\prime).
\end{equation}
Similarly, using \eqref{e:K=RKP}, \eqref{e:PLambda} and \eqref{e:K2}, we can write
\begin{multline}\label{e:K3}
K_{0,x_0}=K_{0,x_0}\circ \mathcal P_{\Lambda} =\sum_{{\mathbf k}^\prime\in\cK_\Lambda}\frac{1}{2^{|{\mathbf k}^\prime|}a^{{\mathbf k}^\prime}{\mathbf k}^\prime!} K_{0,x_0}\circ b^{{\mathbf k}^\prime}\circ \mathcal P\circ (b^+)^{{\mathbf k}^\prime}\\
=\sum_{{\mathbf k},{\mathbf k}^\prime\in\cK_\Lambda}\frac{1}{2^{|{\mathbf k}|+|{\mathbf k}^\prime|}a^{{\mathbf k}+{\mathbf k}^\prime}{\mathbf k}!{\mathbf k}^\prime!} b^{\mathbf k}_zF_{\mathbf k}\cP\circ b^{{\mathbf k}^\prime}\circ \mathcal P\circ (b^+)^{{\mathbf k}^\prime}.
\end{multline}
Now we proceed as follows:
\begin{multline}\label{e:K4}
(F_{\mathbf k}\cP\circ b^{{\mathbf k}^\prime})(Z, Z^\prime)=(\bar{b}^+_z)^{{\mathbf k}^\prime} (F_{\mathbf k}\cP)(Z, Z^\prime)\\ =\sum_{{\mathbf l}\leq {\mathbf k}^\prime}{{\mathbf k}^\prime \choose {\mathbf l} }2^{|{\mathbf l}|}\frac{\partial^{|\mathbf |l}}{\partial z^{\mathbf l}}F_{\mathbf k}(z, Z^\prime)(\bar{b}^+_z)^{{\mathbf k}^\prime-l}\cP(Z, Z^\prime)\\ =\sum_{{\mathbf l}\leq {\mathbf k}^\prime}{{\mathbf k}^\prime \choose{\mathbf l}l}2^{|{\mathbf l}|}\frac{\partial^{|\mathbf l|}}{\partial z^{\mathbf l}}F_{\mathbf k}(z, Z^\prime)(a\bar{z}^\prime)^{{\mathbf k}^\prime-l}\cP(Z, Z^\prime)=F_{{\mathbf k}{\mathbf k}^\prime}\cP(Z, Z^\prime) 
\end{multline}
with some $F_{{\mathbf k}{\mathbf k}^\prime}\in \CC[z, Z^\prime]$.

By \cite[Proof of Lemma 2.2]{ma-ma08a}, there exists $Q_{{\mathbf k}{\mathbf k}^\prime}\in \CC[z, \bar z^\prime]$ such that 
\begin{equation}\label{e:K5}
(F_{{\mathbf k}{\mathbf k}^\prime}\cP\circ \mathcal P)(Z, Z^\prime)=2^{|{\mathbf k}|+|{\mathbf k}^\prime|}a^{{\mathbf k}+{\mathbf k}^\prime}{\mathbf k}!{\mathbf k}^\prime! Q_{{\mathbf k}{\mathbf k}^\prime}\mathcal P(Z, Z^\prime). 
\end{equation}
Combining \eqref{e:K3}, \eqref{e:K4} and \eqref{e:K5}, we complete the proof.
\end{proof}

The following is an analog of an upper estimate for the Wick symbol.  

\begin{prop}\label{p:K0-Wick}
We have 
\[
|K_{0,x_0}(Z,Z)|\leq \limsup_{p\to \infty} \|T_p\|\left(1+\mathcal O(p^{-\frac{1}{2}})\right).
\] 
for any $x_0\in X$ and $Z\in T_{x_0}X$.  
\end{prop}

\begin{proof}
By Definition \ref{d:Toeplitz0} (i), we get
\begin{multline}\label{e:TpLTpL}
T_p(x,x^\prime)=(P_{p,\Lambda}T_pP_{p,\Lambda})(x,x^\prime)\\ =\int P_{p,\Lambda}(x,y) T_p(y,y^\prime) P_{p,\Lambda}(y^\prime,x^\prime) dv_X(y)\,dv_X(y^\prime). . 
\end{multline}
For any $x,y\in X$, $P_{p,\Lambda}(y,x)$ is a linear map from $(L^p\otimes E)_x$ to $(L^p\otimes E)_y$.
For $x\in X$ and $v\in (L^p\otimes E)_x$, introduce $S^p_{x,v}\in C^\infty(X. L^p\otimes E)$ by 
\[
S^p_{x,v}(y)=P_{p,\Lambda}(y,x)v, \quad y\in X.
\]
Observe that 
\begin{align*}
\langle S^p_{x,v}, S^p_{x^\prime,v^\prime}\rangle =& \int_X \langle P_{p,\Lambda}(y,x)v, P_{p,\Lambda}(y,x^\prime)v^\prime\rangle dv_X(y) \\
=& \int_X \langle v, P_{p,\Lambda}(x,y) P_{p,\Lambda}(y,x^\prime)v^\prime\rangle dv_X(y) \\ =&\langle v, P_{p,\Lambda}(x,x^\prime)v^\prime\rangle. 
\end{align*}
In particular, we have
\[
\langle S^p_{x,v}, S^p_{x,v}\rangle =\langle v, P_{p,\Lambda}(x,x)v\rangle= p^n|v|^2\left(\mathcal P_{\Lambda}(0,0)+\mathcal O(p^{-\frac{1}{2}})\right), 
\]
and
\[
\|S^p_{x,v}\|=p^{n/2}|v| \left(\left(\mathcal P_{\Lambda}(0,0)\right)^{1/2}+\mathcal O(p^{-\frac{1}{2}})\right). 
\]

By \eqref{e:TpLTpL}, we infer that  for $x, x^\prime\in X$, $v\in (L^p\otimes E)_x$ and $v^\prime\in (L^p\otimes E)_{x^\prime}$,
\[
\langle v, T_p(x,x^\prime)v^\prime\rangle =\int_X \langle S^p_{x,v}(y), T_p(y,y^\prime)S^p_{x^\prime,v^\prime}(y^\prime)\rangle dv_X(y)\,dv_X(y^\prime)=\langle S^p_{x,v}, T_p S^p_{x^\prime,v^\prime}\rangle. 
\]
It follows that
\begin{equation}
p^{-n}\left|T_p(x,x^\prime)\right| \leq \|T_p\|\left(\mathcal P_{\Lambda}(0,0)+\mathcal O(p^{-\frac{1}{2}})\right).
\end{equation}

Fix $x_0\in X$ and write $x=\exp_{x_0}^X(Z)$ and $x^\prime=\exp_{x_0}^X(Z^\prime)$. Then. by (iii), we get
\[
p^{-n}T_p(x,x^\prime)=p^{-n}T_{p,x_0}(Z,Z^\prime)\cong K_{0,x_0}(\sqrt{p}Z,\sqrt{p}Z^\prime)+\mathcal O(p^{-\frac{1}{2}})
\]
for $|Z|, |Z^\prime|<\varepsilon$. It follows that, for $|Z|, |Z^\prime|<\varepsilon\sqrt{p}$,
\begin{align*}
K_{0,x_0}(Z,Z^\prime) = & p^{-n}T_{p,x_0}\left(\tfrac{1}{\sqrt{p}}Z,\tfrac{1}{\sqrt{p}}Z^\prime\right)+\mathcal O(p^{-\frac{1}{2}})\\ 
=& p^{-n}T_{p}\left(\exp_{x_0}^X\left(\tfrac{1}{\sqrt{p}}Z\right),\exp_{x_0}^X\left(\tfrac{1}{\sqrt{p}}Z\right)\right)+\mathcal O(p^{-\frac{1}{2}})\\ 
 \leq & \|T_p\|\left(1+\mathcal O(p^{-\frac{1}{2}})\right),
\end{align*}
that completes the proof.
\end{proof}

From now on, we will assume that $\mathcal K_\Lambda$ consists of a single element. 
The following is an analog of \cite[Proposition 4.11]{ma-ma08a}. We give a different proof, which is shorter than in \cite{ma-ma08a} and based on Proposition \ref{p:K0-Wick}.

\begin{prop}\label{p:Q00}
Assume that $\mathcal K_\Lambda$ consists of a single element ${\mathbf k}\in \ZZ^n_+$. 
Then 
\[
K_{0,x_0}(Z,Z^\prime)=Q_{x_0}\mathcal P_{\Lambda_{\mathbf k}}(Z,Z^\prime)
\]
for any $x_0\in X$ and $Z,Z^\prime\in T_{x_0}X$ with some $Q_{x_0}\in \operatorname{End}(E_{x_0})$. 
\end{prop}

\begin{proof}
By Proposition \ref{p:K_0zz}. using Leibniz rule, we get
\begin{multline}\label{e:K0-Qkk}
K_{0,x_0}(Z,Z^\prime)\\ =\sum_{{\mathbf l}\leq {\mathbf k}}\sum_{{\mathbf l}^\prime\leq {\mathbf k}}{{\mathbf k} \choose {\mathbf l}}{{\mathbf k} \choose {\mathbf l}^\prime} \frac{\partial^{2{\mathbf k}-{\mathbf l}-{\mathbf l}^\prime}}{\partial z^{{\mathbf k}-{\mathbf l}} \partial z^{\prime {\mathbf k}-{\mathbf l}^\prime}}  Q_{{\mathbf k}{\mathbf k},x_0}(z,\bar z^\prime)b^{\mathbf l}_z{\bar b}^{{\mathbf l}^\prime}_{z^\prime}\mathcal P(Z,Z^\prime).
\end{multline}
To compute $b^{\mathbf l}_z{\bar b}^{{\mathbf l}^\prime}_{z^\prime}\mathcal P$, we write $\mathcal P$ as the product
\[
\mathcal P(Z,Z^\prime)=\prod_{j=1}^n \mathcal P_j(Z_j,Z_j^\prime)
\] 
and treat each factor separately. Then we get for $l_j\geq l^\prime_j$,  
\begin{multline*}
b^{l_j}_{j,z_j}{\bar b}^{l^\prime_j}_{j,z^\prime_j}\mathcal P_j(Z_j,Z_j^\prime)\\ 
=2^{l^\prime_j}a_j^{l_j}l^\prime_j! (\bar z_j-\bar z^\prime_j)^{l_j-l^\prime_j} L^{(l_j-l^\prime_j)}_{l^\prime_j}\left(\frac{a_j|z_j-z^\prime_j|^2}{2}\right)\mathcal P_j(Z_j,Z_j^\prime),
\end{multline*}
where $L^{(m)}_k$, $k,m\in\ZZ_+$, is the generalized Laguerre polynomial:
\[
L^{(m)}_{k}(x)=\frac{x^{-m}e^x}{k!}\frac{d^k}{dx^k}(e^{-x}x^{m+k})=\sum_{j=0}^{k}\binom{k+m}{k-j}\frac{(-x)^j}{j!}, \quad x\geq 0,
\] 
The formula for $l_j\leq l^\prime_j$ is obtained by considering the adjoints. 

In particular, for $l_j=l_j^\prime$, we have 
\begin{equation}\label{e:llP0}
b^{l_j}_{j,z_j}{\bar b}^{l_j}_{j,z^\prime_j}\mathcal P_j(Z_j,Z_j^\prime) =2^{l_j}a_j^{l_j}l_j!L_{l_j}\left(\frac{a_j|z_j-z^\prime_j|^2}{2}\right) \mathcal P_j(Z_j,Z_j^\prime),
\end{equation} 
where $L_k=L^{(0)}_k$, $k\in\ZZ_+$, is the Laguerre polynomial, 
and 
\begin{equation}\label{e:llP}
b^{l_j}_{j,z_j}{\bar b}^{l_j}_{j,z^\prime_j}\mathcal P_j(Z_j,Z_j)=2^{l_j}a_j^{l_j}l_j! \mathcal P_j(0,0). 
\end{equation}
For $l_j\neq l^\prime_j$, we get
\begin{equation}\label{e:llP1}
b^{l_j}_{j,z_j}{\bar b}^{l^\prime_j}_{j,z^\prime_j}\mathcal P_j(Z_j,Z_j)=0. 
\end{equation}

By \eqref{e:llP0} and \eqref{e:PLambda-kernel}, we derive an explicit formula for $\mathcal P_{\Lambda_{\mathbf k}}(Z,Z^\prime)$.
\begin{multline}\label{e:PLambda-k}
\mathcal P_{\Lambda_{\mathbf k}}(Z,Z^\prime)=\frac{1}{(2\pi)^n}\prod_{j=1}^na_j L_{k_j}\left(\frac{a_j|z_j-z^\prime_j|^2}{2}\right)\times \\ \times \exp\left(-\frac 14\sum_{k=1}^na_k(|z_k|^2+|z_k^\prime|^2- 2z_k\bar z_k^\prime) \right).
\end{multline}
 
Taking into account \eqref{e:llP} and \eqref{e:llP1}, the equality \eqref{e:K0-Qkk} for $Z=Z^\prime$ takes the form
\begin{equation*}
K_{0,x_0}(Z,Z)=\sum_{{\mathbf l}\leq {\mathbf k}}2^{|{\mathbf l}|}a^{\mathbf l}{\mathbf l}! {{\mathbf k} \choose {\mathbf l}}^2 \frac{\partial^{2{\mathbf k}-2{\mathbf l}}}{\partial z^{{\mathbf k}-{\mathbf l}} \partial z^{\prime {\mathbf k}-{\mathbf l}}}  Q_{{\mathbf k}{\mathbf k},x_0}(z,\bar z)\mathcal P(0,0).
\end{equation*}
By Proposition \ref{p:K0-Wick}, we get 
\begin{equation*}
\sum_{{\mathbf l}\leq {\mathbf k}}2^{|{\mathbf l}|}a^{\mathbf l}{\mathbf l}! {{\mathbf k} \choose {\mathbf l}}^2 \frac{\partial^{2{\mathbf k}-2{\mathbf l}}}{\partial z^{{\mathbf k}-{\mathbf l}} \partial z^{\prime {\mathbf k}-{\mathbf l}}}  Q_{{\mathbf k}{\mathbf k},x_0}(z,\bar z)=const.
\end{equation*}
Comparing the top degree coefficients in both sides of the last identity, one can easily see that $ Q_{{\mathbf k}{\mathbf k},x_0}(z,\bar z)=Q_{x_0}=const$. Since $ Q_{{\mathbf k}{\mathbf k},x_0}$ is a polynomial of $z$ and $\bar z^\prime$, this implies $ Q_{{\mathbf k}{\mathbf k},x_0}(z,\bar z^\prime)=Q_{x_0}$.
\end{proof}

From now on, we will closely follow the arguments of the proof of Theorem 4.9 in \cite{ma-ma08a}. So we will be brief. 

Define a section $g_0\in C^\infty(X,\operatorname{End}(E))$, setting 
\[
g_0(x_0)= \mathcal Q_{x_0},
\] 
where $\mathcal Q_{x_0}\in \operatorname{End}(E_{x_0})$ is given by Proposition \ref{p:Q00}).

The following is an analog of \cite[Proposition 4.17]{ma-ma08a}. Its proof is based on Proposition~\ref{p:Q00} and therefore a little bit shorter than the proof of \cite[Proposition 4.17]{ma-ma08a}.

\begin{prop}\label{p:Tp-Tg0p}
We have $p^{-n}(T_p-T_{g_0,p})(Z, Z^\prime)\cong \mathcal O(p^{-1})$. 
\end{prop}

\begin{proof}
Consider the sequence of operators 
\[
R_p=p^{1/2}(T_p-T_{g_0,p}), \quad p\in \NN.
\]
It is easy to see that it is in $\mathfrak A$. Moreover, computing the full off-diagonal expansions for the kernels of $T_p$ and $T_{g_0,p}$, we get
\[
p^{-n}R_{p,x_0}(Z, Z^\prime)\cong (K_{1,x_0}-K_{1,x_0}(g_0))(\sqrt{p}Z,\sqrt{p}Z^\prime)+\mathcal O(p^{-1/2}).
\] 
We apply Proposition~\ref{p:Q00} to $\{R_p\}$ and infer that
\[
(K_{1,x_0}-K_{1,x_0}(g_0))(Z,Z^\prime)=S_{x_0}\mathcal P_{\Lambda_{\mathbf k}}(Z,Z^\prime)
\] 
for any $x_0\in X$ and $Z,Z^\prime\in T_{x_0}X$ with some $S_{x_0}\in \operatorname{End}(E_{x_0})$. 

Since $K_{1,x_0}-K_{1,x_0}(g_0)=(\mathcal Q_{1,x_0}-\mathcal Q_{1,x_0}(g_0))\mathcal P$, where $\mathcal Q_{1,x_0}$ and $\mathcal Q_{1,x_0}(g_0)$ are odd, we conclude that 
\[
(K_{1,x_0}-K_{1,x_0}(g_0))(Z,Z^\prime)=0
\] 
for any $x_0\in X$ and $Z,Z^\prime\in T_{x_0}X$, that completes the proof. 
\end{proof}

 By Proposition~\ref{p:Tp-Tg0p}, we have $T_p=P_{p,\Lambda}g_0P_{p,\Lambda}+\mathcal O(p^{-1})$. Consider the operator $p(T_p-P_{p,\Lambda}g_0P_{p,\Lambda})$. It is easy to see that it belongs to $\mathfrak A$. 
By Proposition \ref{p:Q00}, the leading coefficient $K^\prime_{0,x_0}(Z,Z^\prime)$ in the full off-diagonal expansion for the kernel of $p(T_p-P_{p,\Lambda}g_0P_{p,\Lambda})$ has the form
\[
K^\prime_{0,x_0}(Z,Z^\prime)=Q^\prime_{x_0}\mathcal P_{\Lambda_{\mathbf k}}(Z,Z^\prime)
\]
for any $x_0\in X$ and $Z,Z^\prime\in T_{x_0}X$ with some $Q^\prime_{x_0}\in \operatorname{End}(E_{x_0})$. Setting 
\[
g_1(x_0)= \mathcal Q^\prime_{x_0},
\] 
we get a section $g_1\in C^\infty(X,\operatorname{End}(E))$, and, as in Proposition \ref{p:Tp-Tg0p}, we will show that $p(T_p-P_{p,\Lambda}g_0P_{p,\Lambda})=P_{p,\Lambda}g_1P_{p,\Lambda}+\mathcal O(p^{-1})$, which implies $T_p=P_{p,\Lambda}(g_0+p^{-1}g_1)P_{p,\Lambda}+\mathcal O(p^{-2})$. So we can proceed by induction to complete the proof of Theorem \ref{t:characterization}. 

Using Theorems~\ref{t:characterization} and \ref{t:comm}, one can easily complete the proof of Theorem \ref{t:algebra}.

\end{document}